\documentclass[12pt]{amsart}
\usepackage{amsmath,amssymb,amsbsy,amsfonts,amsthm,latexsym,mathabx,
            amsopn,amstext,amsxtra,euscript,amscd,stmaryrd,mathrsfs,
            cite,array,mathtools,enumerate}
\usepackage{url}
\usepackage[colorlinks,linkcolor=blue,anchorcolor=blue,citecolor=blue,backref=page]{hyperref}
\usepackage{color}
\usepackage{graphics,epsfig}
\usepackage{graphicx}
\usepackage{float}

\usepackage{mathtools}
\usepackage{todonotes}

\hypersetup{breaklinks=true}

\RequirePackage{mathrsfs} \let\mathcal\mathscr

\renewcommand*{\backref}[1]{}
\renewcommand*{\backrefalt}[4]{%
    \ifcase #1 (Not cited.)%
    \or        (p.\,#2)%
    \else      (pp.\,#2)%
    \fi}
 
\usepackage[all]{xy}

\newtheorem{theorem}{Theorem}
\newtheorem{thm}[theorem]{Theorem}
\newtheorem{lem}[theorem]{Lemma}

\newtheorem{cor}[theorem]{Corollary}

\numberwithin{equation}{section}
\numberwithin{theorem}{section}
\numberwithin{table}{section}

\newfont{\teneufm}{eufm10}
\newfont{\seveneufm}{eufm7}
\newfont{\fiveeufm}{eufm5}
\newfam\eufmfam
     \textfont\eufmfam=\teneufm
\scriptfont\eufmfam=\seveneufm
     \scriptscriptfont\eufmfam=\fiveeufm
\def\eqref#1{(\ref{#1})}

\def\cG{{\mathcal G}}

\def\cI{{\mathcal I}}

\def\cV{{\mathcal V}}

\def\cX{{\mathcal X}}

\def\({\left(}
\def\){\right)}
\def\l|{\left|}
\def\r|{\right|}

\def\rf#1{\left\lceil#1\right\rceil}

\def\mand{\qquad \mbox{and} \qquad}

\theoremstyle{definition}

 \def \F{{\mathbb F}}

\def \Z{{\mathbb Z}}

	\def \l{\lambda}
	
\begin{document}
\title{Value sets of sparse polynomials}

\author{Igor E. Shparlinski}
\address{School of Mathematics and Statistics, University of New South Wales,
 Sydney NSW 2052, Australia}
\email{igor.shparlinski@unsw.edu.au}

\author{Jos\'e Felipe Voloch}

\address{
School of Mathematics and Statistics,
University of Canterbury,
Private Bag 4800, Christchurch 8140, New Zealand
}
\email{felipe.voloch@canterbury.ac.nz}

\begin{abstract} We obtain a new lower bound on the size of value
set $\cV(f) = f(\F_p)$ of a sparse polynomial $f\in \F_p[X]$ over a finite field of $p$ elements when $p$ is prime. This bound is uniform 
with respect of the degree and depends on 
some natural arithmetic properties of the degrees of the monomial
terms of $f$ and the number of these terms. Our result is stronger than 
those which canted  be extracted from the bounds on multiplicities 
of individual values in $\cV(f)$.
\end{abstract}


\maketitle

\section{Introduction}

The value set of a polynomial $f(X) \in \F_q[X]$ over a finite field $\F_q$ of $q$ elements,  is the set $\cV(f) = \{ f(a):~a \in \F_q \}$ and we define 
$V(f) = \# \cV(f)$. A much studied problem is to estimate
$V(f)$ in terms of $f$. An easy lower bound is $V(f) \ge q/\deg f$ as $f(x)=c$ has at most $\deg f$ solutions for any $c$.
This is essentially best possible in general but, given conditions
on $f$ it can sometimes be improved. In this paper, we  study 
the question of bounding from below $V(f)$ as a function of the number
of terms in $f$, rather than its degree. Specifically, if
$f(X) = a_0+ \sum_{i=1}^t a_iX^{n_i}$, we want to estimate $V(f)$ in terms
of $t$ and $q$. When the degree of $f$ is much higher than $t$, the
polynomial $f$ is said to be sparse. One can
bound the number of roots of sparse polynomials (see~\cite[Lemma~7]{CFKLLS})
and convert this to a lower bound on $V(f)$, as above. Oftentimes,
as described in~\cite{BCR,CGRW} a sparse polynomial may have many roots.
We prove, however, that for $q=p$ prime one can give a nontrivial lower bound on $V(f)$,
for $f$ sparse, 
even when equations of the form $f(x) =a$ have many roots in $\F_p$. 
In addition, this bound is  always better
than the one obtained from the upper bound of~\cite[Lemma~7]{CFKLLS} 
on the number of roots,
when it applies, for $t \ge 9$. 

We obtain our results in three steps. First, using a 
monomial change of variables we reduce the degree of the polynomial, 
as in~\cite{CFKLLS}. Second, we bound the number of irreducible components
of $f(X)-f(Y)$ by adapting a result of Zannier~\cite{Z}. Finally, we use
the results of~\cite{V} to get our bounds. 

We also give a special treatment in the case of binomials, via different arguments
and we obtain stronger results in that case.

%

We recall that the notations $U = O(V)$,  $U \ll V$ and  $V \gg U$,   are all 
equivalent to the statement that $|U| \le cV$ for some constant $c$ which throughout 
this work may depend  on the positive integer parameter $t$
(the number of terms of the polynomials involved), 
and is absolute otherwise.


\section{Factors of differences of sparse Laurent polynomials}

%

We say a polynomial $g(X,Y)$ is a factor of rational function $f(X,Y)$ 
if it is a factor of its numerator (in its lowers terms).

The following result and its proof are motivated by a result in~\cite{Z}.

\begin{thm}
\label{deg-bound}
Let $K$ be a field of positive characteristic $p$ and let 
$$f(X) = \sum_{i=1}^t a_iX^{n_i} \in K(X)$$
be a nonconstant Laurent polynomial over $K$ with $a_i \ne 0$ and nonzero integer
exponents $n_1<\ldots<n_t$ with $n_t \ge |n_i|$, $i=1, \ldots, t$.
If $h(X,Y)$ is an irreducible polynomial factor of $f(X)-f(Y)$ of degree $d$
not of the form $X -\alpha Y$ or $XY-\alpha$, $\alpha \in K$. Then
$d \gg \min \{p/n_t, \sqrt{n_t/t^2}\}$.
\end{thm}

\begin{proof}
Let $\cX$ be a smooth model of the curve $h=0$. The genus of $\cX$ is at most
$(d-1)(d-2)/2$. On $\cX$, the functions $x$ and $y$ have at most $d$ zeros and $d$
poles (on the line at infinity) so they are $S$-units for some set $S$ 
of places of $\cX$ with $\# S \le 3d$.
Consider the functions $x^{n_i},y^{n_i}$, $i=1,\ldots,t$ which are also $S$-units.
Let $u_1=x^{n_t},u_2\ldots,u_m$ be a subset of these functions such that
$$
u_1 = \sum_{i=2}^m c_iu_i, \qquad c_i \in K, 
$$ 
and $m$ minimal. Note that $m \le 2t$ as the equation $f(x)-f(y)=0$ yields a relation of this form with $m=2t$ but may not be
minimal. Note also that $m>1$. 
If $m=2$, then $u_2$ is a power of $y$ as, otherwise $h$ would be a 
polynomial in $X$ which is clearly not possible. Let $u_2=y^{n_j}$.
As, on the curve $h=0$, we have $x^{n_t} = c_2 y^{n_j}$
we must have $n_j \ne 0$ and $y = cx^{n_t/n_j}$ for some $c$ (as algebraic
functions) and plugging this into $f(x)-f(y)=0$ and comparing powers of
$x$ yields $n_j = n_t$ or $n_1$ (the latter only if $n_1=-n_t$)
consequently, $h=X-\alpha Y$ or $h= XY -\alpha$, $\alpha \in K$, 
contrary to hypothesis, so $m \ge 3$. 



The $u_i$ are functions on $\cX$ so 
$$
\(u_1{:}\ldots{:}u_{m-1}\)
$$  
defines a morphism $\cX \to \mathbb{P}^{m-1}$ of
degree at most $3dn_t$, since each coordinate is a
monomial in $x$ or $y$ or their inverses to a power at most $n_t$. If $3dn_t \ge p$, the 
desired result follows immediately.
If $3dn_t < p$,  then~\cite[Theorem~4]{Vabc} holds with the same proof
in characteristic $p>0$ (as the morphism has classical orders by~\cite[Corollary~1.8]{SV}).
Also $\deg u_1 \ge n_t$ so we get
$$
n_t \le \deg u_1 \le (m(m+1)/2)(d(d-3)+3d) \ll d^2 m^2 \le d^2 t^2
$$ 
proving the desired result. 
\end{proof}

\section{Value sets of sparse polynomials}

Here we only concentrate on the case of a prime field $\F_p$, where $p$ is a
prime.

We start with the following simple application of the 
of the Dirichlet pigeonhole principle (see also the proof of~\cite[Lemma~6]{CFKLLS}).

\begin{lem}
\label{lem:Dirichlet} For an integer $S\ge 1$ and arbitrary integers $n_1, \ldots, n_t$, 
there exists a positive integer $s \le S$, such that 
$$
sn_i \equiv m_i \pmod{p-1} \mand  |m_i|  \ll  p S^{-1/t},\qquad i=1,\ldots, t.
$$
\end{lem}

\begin{proof} We cover the cube $[0,p-1]^t$ by at most $S$ cubes with the side length
$ pS^{-1/t}$. Therefore, at least two of the vectors formed by the residues of modulo $p-1$ 
of the $S+1$ vectors $(sn_1, \ldots, sn_t)$, $ s = 0, \ldots, S$ fall in the same cube. 
Assume they correspond to $S \ge s_1 > s_2 \ge 0$. It is easy to see that $s = s_1-s_2$ 
yields the desired result. 
\end{proof} 

Furthermore, by~\cite[Lemma~7]{CFKLLS} we have:

\begin{lem}
\label{lem:SprEq}
For 
$r \ge 2$ given 
elements  $b_1, \ldots, b_r \in \F_p^*$ and
integers $k_1, \ldots, k_r$ in $\Z$
let us denote by $T$ the number of solutions of the equation
$$
\sum_{i=1}^r c_ix^{k_i} = 0, \qquad x \in \F_p^*.
$$
Then
\begin{equation}
\label{SprEqn}
T \le 2 p^{1 - 1/(r-1)} D^{1/(r-1)}
+ O\(p^{1 - 2/(r-1)} D^{2/(r-1)} \),
\end{equation}
where
$$
D =  \min_{1 \le i \le r} \max_{j \ne i} \gcd(k_j - k_i, p-1).
$$
\end{lem}

We also use that by the Cauchy inequality
\begin{equation}
\label{eq:Cauchy}
\begin{split}
p^2 & =\(\sum_{a \in \F_p} \#\{x \in \F_p:~f(x) =a\}\)^2  \\ & \le V(f) \sum_{a \in \cV(f)}\(\#\{x \in \F_p:~f(x) =a\}\)^2 \\ & = V(f) \#\{(x,y)\in \F_p^2:~f(x) =f(y)\}, 
\end{split}
\end{equation}
see also~\cite[Lemma~1]{V}.

We are now ready to estimate $V(f)$.

\begin{thm}
\label{thm:sparse} For any $t \ge 2$ there is a constant $c(t) > 0$ such that for 
any primes  $p$  and integers $1\le n_1, \ldots, n_t<p-1$  
integers with  
\begin{itemize} 
\item[(i)] $\max_{1 \le j < i \le t} \gcd(n_j - n_i, p-1) \le c(t) p$,
\item[(ii)] $ \gcd(n_1, \ldots,  ,n_t, p-1 )=1$, 
\end{itemize}
for any polynomial
$$ 
f(X) = \sum_{i=1}^t a_iX^{n_i} \in \F_p[X] \quad \text{with} \ a_i \ne 0,\ i =1, \ldots, t,
$$
we have $V(f) \gg \min\{p^{2/3}, p^{4/(3t+4)}\}$.
\end{thm}

\begin{proof} We chose the integer parameter 
\begin{equation}
\label{eq:S def}
S =  \rf{p^{3t/(3t+4)}}, 
\end{equation}
and define $s$ and $m_1, \ldots, m_t$ as in Lemma~\ref{lem:Dirichlet}. 

We see from  Lemma~\ref{lem:SprEq} that for a sufficiently small $c(t)$ the 
condition~(i) guarantees that  there is $c \in \F_p^*$ 
such that  
\begin{equation}
\label{eq:Nonvanish}
\sum_{i\in \cI}  a_ic^{n_i} \ne 0,  
\end{equation}
for  all non-empty sets $\cI \subseteq \{1, \ldots, t\}$. 

We now fix some $c \in \F_p^*$ satisfying~\eqref{eq:Nonvanish} and for the above $s$, we
consider the polynomial $f(cX^s)$, then   
the values of $f(cX^s)$ in  $\F_p^*$ coincide with those of 
$$g(X) = \sum_{i=1}^t b_i X^{m_i} \quad \text{with} \  b_i = a_ic^{n_i}, \ i=1,\ldots, t,
$$ 
and, after collecting 
like powers of $X$, we consider two situations: when $g(X)$ is a constant functions
and when $g(X)$ is of positive degree.

We observe that due to the condition~\eqref{eq:Nonvanish} the number of terms
of $g(X)$ is exactly the same as the number of distinct values among 
$m_1, \ldots,  m_t$.

If $g(X)$ is a constant  then $m_1 = \ldots = m_t = 0$  and thus using that 
$sn_i \equiv m_i \equiv  0\pmod{p-1}$, $i =1, \ldots, t$,  we also see that 
$$
s\gcd(n_1, \ldots,  ,n_t, p-1 )   \equiv  0\pmod{p-1}.
$$
This implies that 
$$
S \ge s \ge \frac{p-1}{\gcd(n_1, \ldots,  ,n_t, p-1 )} = p-1
$$
which is impossible for  the above choice of $S$, due to the  condition~(ii).

So we can now assume  that $g(X)$ is a nontrivial Laurent polynomial.

Furthermore, making, if necessary,  the change of variable $X \to X^{-1}$, without loss of generality, we can assume that 
$$
m_t  =  \max\{|m_1|, \ldots, |m_t|\} > 0.
$$

We  now derive  a bound on
$$
N = \#\left\{(x,y) \in \F_p^2:~g(x)=g(y)\right\}, 
$$
which is based on  Theorem~\ref{deg-bound}. 

If  $\sqrt{m_t} \le p/m_t$  then $m_t \le p^{2/3}$  and the result is trivial.
We immediately obtain 
\begin{equation}
\label{eq:N 5/3}
N \ll m_tp \ll p^{5/3}.
\end{equation} 
Hence we now assume that 
\begin{equation}
\label{eq:mt 2/3}
\sqrt{m_t} > p/m_t.
\end{equation}

First, in order to apply Theorem~\ref{deg-bound}, we need to  investigate the factors of $g(X) - g(Y)$ of the form  $X-\alpha Y$ or 
of the form $XY-\alpha$ with $\alpha$ in the algebraic 
closure of $\F_p$.

In fact for an application to $N$ only factors of these forms with $\alpha \in \F_p$ 
are relevant. 

Let   $\cG_s \subseteq \F_p^*$ be the multiplicative subgroup of elements $\alpha \in \F_p$ 
with $\alpha^s = 1$. Note that $\cG_s$ is a subgroup of elements of multiplicative order $\gcd(s,p-1)$, and thus
$$
\# \cG_s  = \gcd(s,p-1).
$$
We show that for some  $\gamma \in \F_p$  factors of $g(X) - g(Y)$ of the form  
$X-\alpha Y$ and $XY-\alpha$ we have $\alpha \in \cG_s$ 
and $\alpha \in \gamma \cG_s$, respectively.

Clearly, if $g(X)-g(Y)$ has a   factor of the form $X-\alpha Y$  then  $g(X)-g(\alpha X)$ is identical to zero.
Since $g(X)$ is not constant, we see that $\alpha\ne0$.
Hence, denoting by $m$ the multiplicative order of  $\alpha$ in $ \F_p^*$ 
we see that by the  condition~(ii) we have
$$
m \mid \gcd(m_1, \ldots, m_t, p-1) = \gcd(s n_1, \ldots, s n_t, p-1) = \gcd(s,p-1). 
$$
Hence $\alpha \in \cG_s$. 

 The factors 
of $g(X)-g(Y)$ of the form $XY-\alpha$, $\alpha \in K$ imply that $g(X)-g(\alpha/X)$
 is identical to zero. This may occur only  if for every $i=1,\ldots, t$
there exists $j =1,\ldots, t$ with $m_i = -m_j$ and $\alpha^{m_i} = b_i/b_j$.
In particular, there is some $\beta \in \F_p^*$ (which may depend on $m_1, \ldots, m_t$)
such that
$$
\alpha^{ \gcd(m_1, \ldots, m_t,p-1)} = \beta 
$$
which puts $\alpha$ in some fixed coset $\cG_s$.
Hence there are at most $s \le S$ such values of $\alpha$  which contribute at most 
\begin{equation}
\label{eq:N0}
N_0 \ll  pS
\end{equation}
to $N$.

We proceed to get an upper
estimate on $N$ and notice that any further contribution to $N$ 
may only come from factors  of $g(X)-g(Y)$
not of the form $X-\alpha Y$ or $XY-\alpha$.

As $m_t \ll p S^{-1/t}$, 
all such factors $h_1, \ldots, h_k$ of $g(X)-g(Y)$ have degree $d_j = \deg h_j$  
for which, then, by Theorem~\ref{deg-bound} and also using~\eqref{eq:mt 2/3} we derive
$$d_j \gg  \min \{p/m_t, \sqrt{m_t/t^2}\} =  p/m_t \ge S^{1/t}, \quad j =1, \ldots, k.
$$ and there are  
$$
k  \le \frac{2 m_t}{ \min\{d_1, \ldots, d_k\}} \ll p S^{-2/t}
$$ 
such factors. 

Let $N_1$ and $N_2$ be contributions to $N$ from the factors $h_j$ of degree
$d_j < p^{1/4}$ and $d_j \ge p^{1/4}$, respectively.

If  a factor $h$ has degree
$d < p^{1/4}$ the number of rational points on $h=0$ is $O(p)$ by the Weil bound (see~\cite{Lor}), 
so those factors all together
contribute 
\begin{equation}
\label{eq:N1}
N_1 \ll   \sum_{\substack{j=1\\ d_j < p^{1/4}}}^k  p \le  k p \ll p^2 S^{-2/t}.
\end{equation}

The factors with degree $d\ge p^{1/4}$
contribute $O(d^{4/3}p^{2/3})$ by~\cite[Theorem~(i)]{V} and, in total they contribute
$$
N_2 = \sum_{\substack{j=1\\ d_j\ge p^{1/4}}}^k d_j^{4/3}p^{2/3} . 
$$
Using the convexity of the function $z \mapsto z^{4/3}$ and then extending the range
of summation to polynomials of all degrees, we obtain
\begin{equation}
\label{eq:N2}
N_2 \le p^{2/3} \( \sum_{ j=1}^k d_j\)^{4/3} \le m_t^{4/3}p^{2/3} \ll p^2 S^{-4/(3t)}.
\end{equation}


Combining~\eqref{eq:N0}, \eqref{eq:N1} and~\eqref{eq:N2} we obtain  
$$
N \ll  pS + p^2 S^{-4/(3t)}
$$
which with the choice $S$ as in~\eqref{eq:S def}, becomes 
\begin{equation}
\label{eq:N 4/3t}
N \ll  p^{(6t+4)/(3t+4)}.
\end{equation}

Combining~\eqref{eq:N 5/3} and~\eqref{eq:N 4/3t} with~\eqref{eq:Cauchy}
we derive the result. 
\end{proof}
 
We now consider the case of binomials in more detail. 
\begin{thm}
\label{thm:bin}
If $f(X) = X +aX^n \in \F_p[X]$ and 
$$
d=\gcd(n,p-1) \mand 
e=\gcd(n-1,p-1),
$$
then $$
V(f) \gg \max\{d,p/d, e, p/e\}.
$$
\end{thm}

\begin{proof} Assume that $d \le p^{1/2}$.
There exists  a positive $r \le (p-1)/d$ with $rn/d \equiv 1 \pmod{(p-1)/d}$ so that
$rn  \equiv d  \pmod{p-1}$. Hence, if $x=u^r$, then $f(x) =g(u)$ where $g(u) = u^r + au^{d}$. 

The equation $g(u)=g(v)$ has degree $\max\{r,d\}$ in $v$ so at most
$$p\max\{r, d\} \le p\max\{(p-1)/d, d\} \le p^2/d
$$ 
solutions, as $d \le p^{1/2}$. 
By~\eqref{eq:Cauchy}, we have $V(f) \gg p^2/pd = p/d$. If $d > p^{1/2}$, note that $d > p/d$.

Now, regardless of the size of $d$, notice that for every $u$  with $u^d =1$
the values $f(u) = u+a$  are pairwise distinct. 
Thus $V(f) \ge d$.

Similarly, there exists $s$ with $s(n-1)/e \equiv 1 \pmod{(p-1)/e}$ so that
$sn  \equiv e+s \pmod{p-1}$. Hence, if $x=u^s$, then $f(x) = h(u)$ where $h(u) =u^s + au^{e+s}$. The equation $h(u)=h(v)$
becomes, with $v=tu$ the same as $u^s + au^{e+s} = t^su^s + au^{e+s}t^{e+s}$ and we get that either $u=0$ or 
$1+ au^e =  t^s + au^{e}t^{e+s}$, which has at most $pe$ solutions. 
By~\eqref{eq:Cauchy}, we have $V(f) \gg p^2/pe = p/e$.
 
Furthermore, notice that for every $u$  with $u^e =c$, where $c$ is a fixed non-zero $e$-th power with $1+ac \ne 0$,
the values $f(u) = u(1+ac)$  are pairwise distinct, and we can also add $f(0) = 0$. 
Thus $V(f) \ge e$.

The result now follows. 
\end{proof} 

We now immediately obtain:

\begin{cor}
\label{thm:bin unif}
If $f(X) = X +aX^n \in \F_p[X]$  then $V(f) \ge  p^{1/2}$.
\end{cor}

%
%

\section{Comments} 

Theorem~\ref{thm:bin} extends, with the same proof, for arbitrary finite
fields. On the other hand, Theorem~\ref{thm:sparse} is false as stated 
for arbitrary finite fields. Indeed, the trace polynomial  
$T(X)= X + X^p + \cdots + X^{p^{t-1}}$ has $T(\F_{p^t}) = \F_p$, so
$V(T) = q^{1/t}$ if $q=p^t$. The trace polynomial can be combined with
a monomial $X^{(q-1)/d}$ for some divisor $d$ to break the linearity of $T(X)$.
 Clearly, for $f(X) = X^{(q-1)/d} + T(X)$ we have $V(f) \le (d+1)p$.

We note that one can use  Lemma~\ref{lem:SprEq} directly in a combination 
with~\eqref{eq:Cauchy}. However, in the best possible scenario this approach can only 
give a lower bound of order $p^{1/(t-1)}$, which is always weaker than that of
Theorem~\ref{thm:sparse} for $t \ge 9$.  

If $p$ is a prime such that $(p-1)/2$ is also prime, then it follows from Theorem~\ref{thm:bin}  that, for $f(X) = X +aX^n$, 
$a \ne 0$, $2 \le n \le p-1$, we have $V(f) \ge (p-1)/2$. It can be proved that equality is
attained if $n = p-2$ and $a$ is
a non-square. In this case the pre-image of non-zero elements of $\F_p$ have zero or two 
elements and the pre-image of zero has three elements. A different example is 
$f(X) = X - X^{(p+1)/2}$, which has $V(f) = (p+1)/2$ and the pre-image of $0$ has $(p+1)/2$ 
elements and other pre-images zero or one elements.

For arbitrary primes, we have the following.
Assume that $d \mid (p-1)$ and consider $f(X) = X + aX^{1+(p-1)/d}$.
Choose $a$, if possible, such that 
$\left( (1+a)/(1+\zeta a) \right)^{(p-1)/d} = \zeta$ for
all $\zeta$ with $\zeta^d = 1$. Then, if $x_1^{(p-1)/d} = 1$ and 
$x_{\zeta} = (1+a)x_1/(1+\zeta a)$, then $x_{\zeta}^{(p-1)/d} = \zeta$ and
$f(x_{\zeta}) = f(x_1)$ and it follows that $V(f) = 1 + (p-1)/d$.

To see when we can find such $a$, let $c_{\zeta}$ be such that
$c_{\zeta}^{(p-1)/d} = \zeta$ with $\zeta^d = 1$. Consider the curve
given by the system of equations $(1+u)/(1+\zeta u) = c_{\zeta}v_{\zeta}^d$
in variables $u$ and $v_{\zeta}$, indexed by $\zeta \ne 1$ with $\zeta^d = 1$. A rational
point with $u= a \ne 0$ provides the necessary $a$. The genus of this
curve is at most $d^d/2$ so by the Weil bound on the number of $\F_p$-rational points on curves (see~\cite{Lor}), there is such a point if $p > d^{2d}$.
This construction succeeds if $d \ll \log p/(\log \log p)$.

We also conclude with posing an question about estimating the image size of
polynomials of the form 
$$
F(X) = \prod_{i=1}^t  \(X^{n_i} + a_i\) 
$$ 
Although most of our technique applies in this case as well,  investigating linear factors 
of $F(c X^s) - F(c Y^s)$  seems to be more complicated.

\section*{Acknowledgements}

 The authors like to thank Domingo  G\'omez-P\'erez for several useful comments.  

During the preparation of this work the first author was  supported   by the ARC Grants~DP170100786 and DP180100201

\bibliographystyle{amsalpha}

\end{document}